\documentclass[12pt]{article}
\usepackage{amsmath,amsthm,amsfonts,amssymb,amsopn,amscd,color,graphicx,mathrsfs}
\usepackage{mathtools}
\usepackage{upgreek}
\usepackage{esint}

\newcommand{\vect}[1]{\boldsymbol{#1}}

\newtheorem{theorem}{Theorem}[section]

\newtheorem{defn}[theorem]{Definition}

\newtheorem{lemma}[theorem]{Lemma}

\newtheorem{remark}[theorem]{Remark}

\newcommand{\R}{\Bbb{R}}
\newcommand{\N}{\Bbb{N}}

\newenvironment{pfthm1}{{\par\noindent\bf
Proof of Theorem~\ref{thm1} }}{\hfill\fbox{}\par\vspace{.2cm}}

\title{On Liouville type theorems in the stationary non-Newtonian fluids}

\begin{document}
\renewcommand{\theequation}{\thesection.\arabic{equation}}
\author{Dongho Chae\footnote{E-mail: dchae@cau.ac.kr } , Junha Kim\footnote{E-mail: jha02@cau.ac.kr }   and J\"{o}rg Wolf\footnote{E-mail: jwolf2603@gmail.com }  \,\,\\
Department of Mathematics,\\Chung-Ang University\\
 Seoul 06974, Republic of Korea \\
  }
\date{}
\maketitle
\begin{abstract}
In this paper we prove a Liouville type theorem  for  the stationary   equations   
of  a  non-Newtonian fluid in $\R^3$ with the viscous part of the stress tensor 
$\vect{A}_p(u) = \mathrm{div} ( | \vect{D}(u) |^{p-2} \vect{D}(u) )$, where $\vect{D}(u) = \frac 12 ( \nabla u + ( \nabla u )^{\top})$ and $\frac 95 < p < 3$. We consider a weak solution $u \in W^{1,p}_{loc}(\R^3)$ 
and its potential function $\vect{V} = (V_{ij}) \in W^{2,p}_{loc}(\R^3)$, i.e. $\nabla \cdot \vect{V} = u$.
We show that there exists a constant $s_0=s_0(p)$ such that if  the $L^s$ mean oscillation of $\vect{V}$ for $s>s_0$  satisfies a  certain growth condition at infinity, then the velocity field vanishes.  Our result includes the previous results \cite{CW20, CW19} as particular cases. \\
 \ \\
\noindent{\bf AMS Subject Classification Number:}
35Q35, 35Q30, 76A05, 76D05, 76D03\\
  \noindent{\bf
keywords:} non-Newtonian fluid equations, Liouville type theorem  
\end{abstract}
 
\section{Introduction}
We consider a power law model of non-Newtonian fluid in $\R^3$
\begin{equation}\label{NNFE}
\left\{ \aligned &
- \vect{A}_p(u) + (u \cdot \nabla) u = - \nabla \uppi \quad \mbox{in} \quad \R^3,\\
&\quad\mathrm{div}\, u = 0,
\endaligned \right.
\end{equation}
where $u = (u_1(x), u_2(x), u_3(x))$ is the velocity field, $\uppi = \uppi(x)$ is the pressure field. The diffusion term is represented by
$$
\vect{A}_p(u) = \mathrm{div} ( | \vect{D}(u) |^{p-2} \vect{D}(u) ), \qquad 1 < p < + \infty
$$
and the deviatoric stress tensor is interpreted as $|\vect{D}|^{p-2} \vect{D} = \vect{\upsigma}(\vect{D})$, where $\vect{D}(u) = \frac 12 ( \nabla u + ( \nabla u )^{\top})$ is the symmetric gradient. If $2 < p < + \infty$, the equations describe shear thickening fluids, of which viscosity increases along with shear rate $|\vect{D}(u)|$. If $1 < p < 2$, shear thinning fluids satisfy them. In the case of $p = 2$, \eqref{NNFE} corresponds to the usual stationary Navier-Stokes equations which represent Newtonian fluids. We refer to Wilkinson \cite{W60} for continuum mechanical background of the above system. 

The Liouville problem for the stationary Navier-Stokes equations (Galdi \cite{G11}, Remark X. 9.4, pp. 729) has attracted considerable attention in the mathematical fluid mechanics. Though it is still open, there are positive answers under additional conditions (see \cite{CY13,C14,CW16,CJL19,GW78,KNSS09,KPR15,KTW17,S16,S18,SW19}). And as a generalization, Liouville type theorems for non-Newtonian fluids have been investigated (see \cite{JK14, CW20}). 

Let $u \in L^1_{loc}(\R^3)$ be a vector field, and let $\vect{V} = (V_{ij}) \in L^1_{loc}(\R^3;\R^3 \times \R^3)$ be a matrix valued function satisfying $\mathrm{div}\, \vect{V} = u$ in the distributional sense. In \cite{CW19} Chae and Wolf proved Liouville type theorem for the stationary Navier-Stokes equations when the following is assumed
$$
\left( \fint_{B(r)} |\vect{V} - \vect{V}_{B(r)} |^s \mathrm{d}x \right)^{\frac 1s} \leq C r^{\min \left\{ \frac 13 - \frac 1s, \frac 16 \right\}} \qquad \forall  1 < r < + \infty
$$
for some $3 < s < + \infty$. They also considered it in \eqref{NNFE} when $\frac 95 < p < 3$ but only for $s = \frac {3p}{2p - 3}$ in \cite{CW20}. We generalize these results.

As is well known, weak solutions are actually smooth for $p = 2$. Otherwise there is only partial regularity of weak solutions \cite{P96,FMS03}. In this paper, we consider weak solutions, which is defined as follows :

\begin{defn}\label{WS}
Let $\frac 95 < p < 3$. A function $u \in W^{1,p}_{loc}(\R^3)$ is called a weak solution to \eqref{NNFE} if
\begin{equation}\label{WF1}
\int_{\R^3} \left( | \vect{D}(u) |^{p-2} \vect{D}(u) - u \otimes u \right) : \vect{D}(\varphi) \mathrm{d}x = 0
\end{equation}
is fulfilled for all vector fields $\varphi \in C^{\infty}_c(\R^3)$ with $\mathrm{div}\, \varphi = 0$. 
\end{defn}

\begin{remark}
Let $u$ be a weak solution. Then, there exists $\uppi \in L^{\frac p{p - 1}}_{loc}(\R^3)$ satisfying
\begin{equation}\label{PE}
\int_{B(r)} |\uppi - \uppi_{B(r)}|^s \mathrm{d}x \leq C \int_{B(r)} | | \vect{D}(u) |^{p-2} \vect{D}(u) - u \otimes u|^s \mathrm{d}x, \quad \forall 0 < r < + \infty
\end{equation}
for all $\frac 32 \leq s \leq \frac p{p - 1}$. And $(u, \uppi)$ holds
\begin{equation}\label{WF2}
\int_{\R^3} \left( | \vect{D}(u) |^{p-2} \vect{D}(u) - u \otimes u \right) : \vect{D}(\varphi) \mathrm{d}x = \int_{\R^3} \uppi \mathrm{div}\, \varphi \mathrm{d}x
\end{equation}
for any vector field $\varphi \in W^{1,p}(\R^3)$ with compact support. Hence, \eqref{WF2} { replaces} \eqref{WF1}. We refer to \cite{CW20} for a brief explanation.
\end{remark}

\begin{remark}
Let $(u, \uppi)$ be a weak solution and $\phi \in C^{\infty}_c(\R^3)$. If we take $\varphi = u \phi$, then \eqref{WF2} with $\varphi \in W^{1,p}(\R^3)$ yields the local energy equality
\begin{equation}\label{LEI}
\int_{\R^3} | \vect{D}(u) |^p \phi \mathrm{d}x= - \int_{\R^3} | \vect{D}(u) |^{p - 2} \vect{D}(u) : u \otimes \nabla \phi \mathrm{d}x + \int_{\R^3} \left( \frac 12 | u |^2 + \uppi \right) u \cdot \nabla \phi \mathrm{d}x.
\end{equation}
\end{remark}

\begin{theorem}\label{thm1}
Let $\frac 95 < p < 3$, and $\frac 32 < s < +\infty$ satisfy
$$
s \geq \frac {3p}{2(2p - 3)}, \qquad s > \frac {9 - 3p}{2p - 3}.
$$
Let $(u,\uppi) \in W^{1,p}_{loc}(\R^3) \times L^{\frac p{p - 1}}_{loc}(\R^3)$ be a weak solution to \eqref{NNFE}. We set 
$$
\alpha(p,s) := \min\left\{ \frac 13 - \frac {5p - 9}{s(2p - 3)}, \frac 3p - \frac 43 \right\}.
$$
If we assume there exists a potential $\vect{V} \in W^{2,p}_{loc}(\R^3; \R^{3\times3})$ such that
\begin{equation}\label{AS}
\left( \fint_{B(r)} |\vect{V} - \vect{V}_{B(r)} |^s \mathrm{d}x \right)^{\frac 1s} \leq C r^{\alpha(p,s)} \qquad \forall  1 < r < + \infty,
\end{equation}
then $u \equiv 0$.
\end{theorem}

\begin{remark}
If $s = \frac {3p}{2p-3}$, then we obtain $\alpha = \frac {9 - 4p}{3p}$, which consistent with Theorem 1.3 (ii) in \cite{CW20}.
\end{remark}

\begin{remark}
 In the case of $p=2$, we have $3 < s < + \infty$ and $\alpha = \min \{\frac s3 - 1, \frac s6\}$, corresponding to Theorem 1.1 in \cite{CW19}.
\end{remark}

We denote by $C(p, s) = C$ a generic constant that may vary from line to line. 

\section{Caccioppoli type inequalities}

To prove Theorem~\ref{thm1}, we need the following two lemmas.
\begin{lemma}\label{LEMPP}
Let $\frac 95 < p < 3$. Let $u \in W^{1,p}_{loc}(\R^3)$ and $1 \leq R < + \infty$. For $0 < \rho < R$, we let $\psi \in C^\infty_c(B(R))$ satisfy $0 \leq \psi \leq 1$ and $| \nabla \psi | \leq C (R - \rho)^{-1}$. If we assume there exists $\vect{V} \in W^{2,p}_{loc}(\R^3)$ such that $\mathrm{div}\, \vect{V} = u$ and \eqref{AS} for some $\frac {3p}{2(2p - 3)} \leq s \leq \frac {3p}{2p - 3}$, then for $s \geq p$
\begin{equation}\label{PPPS}
\int_{B(R)} |\psi^p u|^p \mathrm{d}x \leq C R^{\frac 32 + \frac p6 - \frac {p(5p - 9)}{2s(2p - 3)}} \| \psi \nabla u \|_{L^p}^{\frac p2} + C (R - \rho)^{-p} R^{3 + \frac p3 - \frac {p(5p - 9)}{s(2p - 3)}},
\end{equation}
and for $s < p$
$$
\int_{B(R)} |\psi^p u|^p \mathrm{d}x \leq C R^{\frac {p^2}{3(2p - 3)}} \| \psi \nabla u \|_{L^p}^{\frac {p(sp - 3s + 3p)}{2sp - 3s + 3p}} + C (R - \rho)^{- \frac {3p}{2s} + \frac 32} R^{\frac p6 + \frac {p^2}{2s(2p - 3)}} \| \psi \nabla u \|_{L^p}^{\frac p2}
$$
\begin{equation}\label{PPSP}
+ C (R - \rho)^{- \frac {sp^2}{sp + 3p - 3s}} R^{\frac {p^2(2sp + 3p - 3s)}{3(2p - 3)(sp + 3p - 3s)}} \| \psi \nabla u \|_{L^p}^{\frac {p(3p - 3s)}{sp - 3s + 3p}} + C (R - \rho)^{- p - \frac {3p}s + 3} R^{\frac p3 + \frac {p^2}{s(2p - 3)}}.
\end{equation}
\end{lemma}

\begin{proof}
We first let $\frac 95 < p < 2$. Beacuse $s \geq \frac {3p}{2(2p - 3)} \geq p$, we show \eqref{PPPS} in this case. H\"{o}lder's inequality implies that
\begin{equation}\label{PP}
\int_{B(R)} |\psi^p u|^p \mathrm{d}x \leq | B(R) |^{1 - \frac p2} \left( \int_{B(R)} |\psi^p u|^2 \mathrm{d}x \right)^{\frac p2} = C R^{3 - \frac {3p}2} \left( \int_{B(R)} | u |^2 \psi^{2p} \mathrm{d}x \right)^{\frac p2}.
\end{equation}
Recalling $u = \mathrm{div}\, \vect{V}$ and using integration by parts, we have
$$
\int_{B(R)} | u |^2 \psi^{2p} \mathrm{d}x = \int_{B(R)} \partial_i \left(\vect{V}_{ij} - (\vect{V}_{ij})_{B(R)} \right) u_j \psi^{2p} \mathrm{d}x
$$
$$
= - \int_{B(R)} \left(\vect{V}_{ij} - (\vect{V}_{ij})_{B(R)} \right) \partial_i u_j \psi^{2p} \mathrm{d}x - 2p \int_{B(R)} \left(\vect{V}_{ij} - (\vect{V}_{ij})_{B(R)} \right) u_j \psi^{2p - 1} \partial_i \psi \mathrm{d}x
$$
$$
\leq \int_{B(R)} | \vect{V} - \vect{V}_{B(R)} | | \psi \nabla u | \psi^{2p - 1} \mathrm{d}x + 2p \int_{B(R)} | \vect{V} - \vect{V}_{B(R)} | | \psi^p u | | \nabla \psi | \psi^{p - 1} \mathrm{d}x
$$
$$
\leq \int_{B(R)} | \vect{V} - \vect{V}_{B(R)} | | \psi \nabla u | \mathrm{d}x + C (R - \rho)^{-1} \int_{B(R)} | \vect{V} - \vect{V}_{B(R)} | | \psi^p u | \mathrm{d}x.
$$
Note that $\frac {3p}{2(2p - 3)} \leq s \leq \frac {3p}{2p - 3}$ implies
\begin{equation*}
\frac 1s + \frac 1p < 1
\end{equation*}
and
$$
\alpha = \frac 13 - \frac {5p - 9}{s(2p - 3)}.
$$
{ Using} H\"{o}lder's inequality and \eqref{AS}, { we obtain}
$$
\int_{B(R)} | u |^2 \psi^{2p} \mathrm{d}x \leq \left( \fint_{B(R)} | \vect{V} - \vect{V}_{B(R)} |^s \mathrm{d}x \right)^{\frac 1s} \| \psi \nabla u \|_{L^p} | B(R) |^{1 - \frac 1p}
$$
$$
+ C (R - \rho)^{-1} \left( \fint_{B(R)} | \vect{V} - \vect{V}_{B(R)} |^s \mathrm{d}x \right)^{\frac 1s} \| \psi^p u \|_{L^p} | B(R) |^{1 - \frac 1p}
$$
$$
\leq C R^{\frac {10}3 - \frac {5p - 9}{s(2p - 3)} - \frac 3p} \| \psi \nabla u \|_{L^p} + C (R - \rho)^{-1} R^{\frac {10}3 - \frac {5p - 9}{s(2p - 3)} - \frac 3p} \| \psi^p u \|_{L^p}.
$$
Inserting { this inequality}  into \eqref{PP} and applying Young's inequality, we have
$$
\int_{B(R)} |\psi^p u|^p \mathrm{d}x \leq C R^{\frac 32 + \frac p6 - \frac {p(5p - 9)}{2s(2p - 3)}} \| \psi \nabla u \|_{L^p}^{\frac p2} + C (R - \rho)^{- \frac p2} R^{\frac 32 + \frac p6 - \frac {p(5p - 9)}{2s(2p - 3)}} \| \psi^p u \|_{L^p}^{\frac p2}
$$
$$
\leq C R^{\frac 32 + \frac p6 - \frac {p(5p - 9)}{2s(2p - 3)}} \| \psi \nabla u \|_{L^p}^{\frac p2} + C (R - \rho)^{-p} R^{3 + \frac p3 - \frac {p(5p - 9)}{s(2p - 3)}} + \frac 12 \int_{B(R)} |\psi^p u|^p \mathrm{d}x,
$$
which implies \eqref{PPPS}.

Now, we let $2 \leq p < 3$. Using $u = \mathrm{div}\, \vect{V}$, integration by parts and H\"{o}lder's inequality, we find
$$
\int_{B(R)} |\psi^p u|^p \mathrm{d}x = \int_{B(R)} \partial_i \left(\vect{V}_{ij} - (\vect{V}_{ij})_{B(R)} \right) u_j | u |^{p - 2} \psi^{p^2} \mathrm{d}x
$$
$$
= - \int_{B(R)} \left(\vect{V}_{ij} - (\vect{V}_{ij})_{B(R)} \right) \left( \partial_i u_j | u |^{p - 2} + (p-2) u_j u_k \partial_i u_k | u |^{p - 4} \right) \psi^{p^2} \mathrm{d}x
$$
$$
 - p^2 \int_{B(R)} \left(\vect{V}_{ij} - (\vect{V}_{ij})_{B(R)} \right) u_j | u |^{p - 2} \psi^{p^2 - 1} \partial_i \psi \mathrm{d}x
$$
$$
\leq C \int_{B(R)} | \vect{V} - \vect{V}_{B(R)} | | \psi \nabla u | | \psi^{p + 1} u |^{p - 2} \mathrm{d}x 
$$
$$
+ C (R - \rho)^{-1} \int_{B(R)} | \vect{V} - \vect{V}_{B(R)} | | \psi^{p + 1} u |^{p - 1} \mathrm{d}x
$$
$$
\leq C \| \vect{V} - \vect{V}_{B(R)} \|_{L^s(B(R))} \| \psi \nabla u \|_{L^p} \| | \psi^{p + 1} u |^{p-2} \|_{L^{\frac {sp}{sp-s-p}}}
$$
\begin{equation*}
+ C (R - \rho)^{-1} \| \vect{V} - \vect{V}_{B(R)} \|_{L^s(B(R))} \| | \psi^{p + 1} u |^{p-1} \|_{L^{\frac s{s-1}}} = I + II.
\end{equation*}

We assume $s \geq p$ first. Since
$$
1 \leq \frac {sp}{sp-s-p} \leq \frac p{p-2}
$$
by applying H\"{o}lder's inequality to $I$ we have
$$
I \leq C \| \vect{V} - \vect{V}_{B(R)} \|_{L^s(B(R))} \| \psi \nabla u \|_{L^p} \| \psi^p u \|_{L^p}^{p - 2} | B(R) |^{\frac1p - \frac 1s}.
$$
Subsequently, we use \eqref{AS} and Young's inequality. This { yields}
$$
I \leq C R^{\frac 3p} \left( \fint_{B(R)} |\vect{V} - \vect{V}_{B(R)} |^s \mathrm{d}x \right)^{\frac 1s} \| \psi \nabla u \|_{L^p} \| \psi^p u \|_{L^p}^{p-2}
$$
$$
\leq C R^{\frac 3p + \frac 13 - \frac {5p - 9}{s(2p - 3)}} \| \psi \nabla u \|_{L^p} \| \psi^p u \|_{L^p}^{p-2}
$$
\begin{equation}\label{PPI1}
\leq C R^{\frac 32 + \frac p6 - \frac {p(5p - 9)}{2s(2p - 3)}} \| \psi \nabla u \|_{L^p}^{\frac p2} + \frac 14 \int |\psi^p u|^p \mathrm{d}x.
\end{equation}
{ Arguing  similarly to the above, having}
$$
1 < \frac s{s-1} \leq \frac p{p-1},
$$
we can calculate $II$ as follows
$$
II \leq C (R - \rho)^{-1} \| \vect{V} - \vect{V}_{B(R)} \|_{L^s(B(R))} \| \psi^{p + 1} u \|_{L^p}^{p-1} | B(R) |^{\frac 1p - \frac 1s}
$$
$$
\leq C (R - \rho)^{-1} R^{\frac 3p} \left( \fint_{B(R)} |\vect{V} - \vect{V}_{B(R)} |^s \mathrm{d}x \right)^{\frac 1s} \| \psi^p u \|_{L^p}^{p-1}
$$
$$
\leq C (R - \rho)^{-1} R^{\frac 3p + \frac 13 - \frac {5p - 9}{s(2p - 3)}} \| \psi^p u \|_{L^p}^{p-1}
$$
\begin{equation}\label{PPII1}
\leq C (R - \rho)^{-p} R^{3 + \frac p3 - \frac {p(5p - 9)}{s(2p - 3)}} + \frac 14 \int |\psi^p u|^p \mathrm{d}x.
\end{equation}
\eqref{PPI1} and \eqref{PPII1} show \eqref{PPPS}.

Now we assume $s < p$. { Since it holds}
$$
\frac p{p-2} < \frac {sp}{sp-s-p} \leq \frac {3p}{3-p} < \frac {3p}{(3-p)(p-2)},
$$
the standard interpolation inequality implies that
$$
\| | \psi^{p + 1} u |^{p-2} \|_{L^{\frac {sp}{sp-s-p}}} \leq \| | \psi^{p + 1} u |^{p-2} \|_{L^{\frac {3p}{(3-p)(p-2)}}}^{\frac {3p - 3s}{sp(p - 2)}} \| | \psi^{p + 1} u |^{p-2} \|_{L^{\frac p{p-2}}}^{\frac {sp(p - 2) - 3p + 3s}{sp(p - 2)}}
$$
$$
= \| \psi^{p + 1} u \|_{L^{\frac {3p}{3-p}}}^{\frac {3p - 3s}{sp}} \| \psi^{p + 1} u \|_{L^p}^{\frac {sp(p - 2) - 3p + 3s}{sp}}.
$$
{ Thus,} by using Sobolev inequality here, { we arrive at}
$$
\| | \psi^{p + 1} u |^{p-2} \|_{L^{\frac {sp}{sp-s-p}}} \leq C \left( \| \psi^{p + 1} \nabla u \|_{L^p} + (R - \rho)^{-1} \| \psi^p u \|_{L^p} \right)^{\frac {3p - 3s}{sp}} \| \psi^{p + 1} u \|_{L^p}^{\frac {sp(p - 2) - 3p + 3s}{sp}}
$$
$$
\leq C \| \psi \nabla u \|_{L^p}^{\frac {3p - 3s}{sp}} \| \psi^p u \|_{L^p}^{\frac {sp(p - 2) - 3p + 3s}{sp}} + C (R - \rho)^{- \frac {3p - 3s}{sp}} \| \psi^p u \|_{L^p}^{p - 2}.
$$
{ Inserting the inequality we just obtained  into} $I$ and applying \eqref{AS} along with Young's inequality, { we find}
$$
I \leq C R^{\frac 3s} \left( \fint_{B(R)} |\vect{V} - \vect{V}_{B(R)} |^s \mathrm{d}x \right)^{\frac 1s} \| \psi \nabla u \|_{L^p}^{\frac {sp + 3p - 3s}{sp}} \| \psi^p u \|_{L^p}^{\frac {sp(p - 2) - 3p + 3s}{sp}}
$$
$$
+ C (R - \rho)^{- \frac {3p - 3s}{sp}}  R^{\frac 3s} \left( \fint_{B(R)} |\vect{V} - \vect{V}_{B(R)} |^s \mathrm{d}x \right)^{\frac 1s} \| \psi \nabla u \|_{L^p} \| \psi^p u \|_{L^p}^{p - 2}
$$
$$
\leq C R^{\frac 3s + \frac 13 - \frac {5p - 9}{s(2p - 3)}} \| \psi \nabla u \|_{L^p}^{\frac {sp + 3p - 3s}{sp}} \| \psi^p u \|_{L^p}^{\frac {sp(p - 2) - 3p + 3s}{sp}}
$$
$$
+ C (R - \rho)^{- \frac {3p - 3s}{sp}} R^{\frac 3s + \frac 13 - \frac {5p - 9}{s(2p - 3)}} \| \psi \nabla u \|_{L^p} \| \psi^p u \|_{L^p}^{p - 2}
$$
$$
\leq C R^{\frac {p^2}{3(2p - 3)}} \| \psi \nabla u \|_{L^p}^{\frac {p(sp + 3p - 3s)}{2sp + 3p - 3s}} + C (R - \rho)^{-\frac {3p}{2s} + \frac 32} R^{\frac {3p}{2s} - \frac 32 + \left( \frac 32 + \frac p6 - \frac {p(5p - 9)}{2s(2p - 3)} \right)} \| \psi \nabla u \|_{L^p}^{\frac p2} 
$$
\begin{equation}\label{PPI2}
+ \frac 14 \int |\psi^p u|^p \mathrm{d}x.
\end{equation}
We note that $s < p$ implies
$$
\frac p{p-1} < \frac s{s-1} < \frac {3p}{2(3-p)} < \frac {3p}{(3-p)(p-1)}.
$$
Thus, using Interpolation inequality and Sobolev inequality, we infer
$$
\| | \psi^{p + 1} u |^{p - 1} \|_{L^{\frac s{s - 1}}} \leq \| | \psi^{p + 1} u |^{p-1} \|_{L^{\frac {3p}{(3-p)(p-1)}}}^{\frac {3p - 3s}{sp(p - 1)}} \| | \psi^{p + 1} u |^{p-1} \|_{L^{\frac p{p-1}}}^{\frac {sp(p - 1) - 3p + 3s}{sp(p - 1)}}
$$
$$
= \| \psi^{p + 1} u \|_{L^{\frac {3p}{3-p}}}^{\frac {3p - 3s}{sp}} \| \psi^{p + 1} u \|_{L^p}^{\frac {sp(p - 1) - 3p + 3s}{sp}}
$$
$$
\leq C \left( \| \psi^{p + 1} \nabla u \|_{L^p} + (R - \rho)^{-1} \| \psi^p u \|_{L^p} \right)^{\frac {3p - 3s}{sp}} \| \psi^{p + 1} u \|_{L^p}^{\frac {sp(p - 1) - 3p + 3s}{sp}}
$$
$$
\leq C \| \psi \nabla u \|_{L^p}^{\frac {3p - 3s}{sp}} \| \psi^p u \|_{L^p}^{\frac {sp(p - 1) - 3p + 3s}{sp}} + C (R - \rho)^{- \frac {3p - 3s}{sp}} \| \psi^p u \|_{L^p}^{p - 1}.
$$
{ Inserting this inequality } into $II$ and using \eqref{AS} and { Young's} inequality,  { it follows}
$$
II \leq C (R - \rho)^{-1} R^{\frac 3s} \left( \fint_{B(R)} |\vect{V} - \vect{V}_{B(R)} |^s \mathrm{d}x \right)^{\frac 1s} \| \psi \nabla u \|_{L^p}^{\frac {3p - 3s}{sp}} \| \psi^p u \|_{L^p}^{\frac {sp(p - 1) - 3p + 3s}{sp}}
$$
$$
+ C (R - \rho)^{- 1 - \frac {3p - 3s}{sp}}  R^{\frac 3s} \left( \fint_{B(R)} |\vect{V} - \vect{V}_{B(R)} |^s \mathrm{d}x \right)^{\frac 1s} \| \psi^p u \|_{L^p}^{p - 1}
$$
$$
\leq C (R - \rho)^{-1} R^{\frac 3s + \frac 13 - \frac {5p - 9}{s(2p - 3)}} \| \psi \nabla u \|_{L^p}^{\frac {3p - 3s}{sp}} \| \psi^p u \|_{L^p}^{\frac {sp(p - 1) - 3p + 3s}{sp}}
$$
$$
+ C (R - \rho)^{- 1 - \frac {3p - 3s}{sp}} R^{\frac 3s + \frac 13 - \frac {5p - 9}{s(2p - 3)}} \| \psi^p u \|_{L^p}^{p - 1}
$$
$$
\leq C (R - \rho)^{- \frac {sp^2}{sp + 3p - 3s}} R^{\frac {p^2(2sp + 3p - 3s)}{3(2p - 3)(sp + 3p - 3s)}} \| \psi \nabla u \|_{L^p}^{\frac {p(3p - 3s)}{sp + 3p - 3s}}
$$
\begin{equation}\label{PPII2}
+ C (R - \rho)^{- p - \frac {3p}s + 3} R^{\frac {3p}s + \frac p3 - \frac {p(5p - 9)}{s(2p - 3)}} + \frac 14 \int |\psi^p u|^p \mathrm{d}x.
\end{equation}
By \eqref{PPI2} and \eqref{PPII2} we complete the proof.
\end{proof}

\begin{lemma}\label{LEM3P}
Let $\frac 95 < p < 3$. Let $u \in W^{1,p}_{loc}(\R^3)$ and $1 \leq R < + \infty$. For $0 < \rho < R$, we let $\psi \in C^\infty_c(B(R))$ satisfy $0 \leq \psi \leq 1$ and $| \nabla \psi | \leq C (R - \rho)^{-1}$. If we assume there exists $\vect{V} \in W^{2,p}_{loc}(\R^3)$ such that $\mathrm{div}\, \vect{V} = u$ and \eqref{AS} for some $\frac {3p}{2(2p - 3)} \leq s \leq \frac {3p}{2p - 3}$, then for $s \geq 3$
$$
\int_{B(R)} |\psi^3 u|^3 \mathrm{d}x \leq C R \| \psi \nabla u \|_{L^p}^{\frac {9p}{2sp + 3p - 3s}} + C R \left( (R - \rho)^{-1} \| \psi^p u \|_{L^p} \right)^{\frac {9p}{2sp + 3p - 3s}}
$$
\begin{equation}\label{3P3S}
+ C (R - \rho)^{-3} R^{4 - \frac {3(5p - 9)}{s(2p - 3)}},
\end{equation}
and for $s < 3$
$$
\int_{B(R)} |\psi^3 u|^3 \mathrm{d}x \leq C R \| \psi \nabla u \|_{L^p}^{\frac {9p}{2sp + 3p - 3s}} + C R \left( (R - \rho)^{-1} \| \psi^p u \|_{L^p} \right)^{\frac {9p}{2sp + 3p - 3s}}
$$
$$
+ C (R - \rho)^{-\frac {3s(2p - 3)}{sp + 3p - 3s}} R^{\frac {2sp + 3p - 3s}{sp + 3p - 3s}} \| \psi \nabla u \|_{L^p}^{\frac {3p(3 - s)}{sp + 3p - 3s}}
$$
\begin{equation}\label{3PS3}
+ C (R - \rho)^{-\frac {3s(2p - 3)}{sp + 3p - 3s}} R^{\frac {2sp + 3p - 3s}{sp + 3p - 3s}} \left( (R - \rho)^{-1} \| \psi^p u \|_{L^p} \right)^{\frac {3p(3 - s)}{sp + 3p - 3s}}.
\end{equation}
\end{lemma}

\begin{proof}
{ Arguing similarly to  the proof of Lemma \ref{LEMPP}, we get}

$$
\int_{B(R)} |\psi^3 u|^3 \mathrm{d}x 
\leq C \| \vect{V} - \vect{V}_{B(R)} \|_{L^s(B(R))} \| \psi \nabla u \|_{L^p} \| \psi^4 u \|_{L^{\frac {sp}{sp-s-p}}}
$$
\begin{equation*}
+ C (R - \rho)^{-1} \| \vect{V} - \vect{V}_{B(R)} \|_{L^s(B(R))} \| | \psi^4 u |^2 \|_{L^{\frac s{s-1}}} = I + II.
\end{equation*}
{ Firstly, we estimate $I$}. { Since}
$$
3 \leq \frac {sp}{sp - s - p} \leq \frac {3p}{3-p},
$$
Gagliardo–Nirenberg interpolation inequality and H\"{o}lder's inequality imply that
$$
\| \psi^4 u \|_{L^{\frac {sp}{sp - s - p}}} 
\leq C \left( \| \psi^4 \nabla u \|_{L^p} + (R - \rho)^{-1} \| \psi^3 u \|_{L^p} \right)^{\frac {3p + 3s - 2sp}{s(2p - 3)}} \| \psi^4 u \|_{L^3}^{\frac {4sp - 3p - 6s}{s(2p - 3)}}
$$
$$
\leq C \| \psi \nabla u \|_{L^p}^{\frac {3p + 3s - 2sp}{s(2p - 3)}} \| \psi^3 u \|_{L^3}^{\frac {4sp - 3p - 6s}{s(2p - 3)}} + C \left( (R - \rho)^{-1} \| \psi^p u \|_{L^p} \right)^{\frac {3p + 3s - 2sp}{s(2p - 3)}} \| \psi^3 u \|_{L^3}^{\frac {4sp - 3p - 6s}{s(2p - 3)}}.
$$
We insert it into $I$ and use \eqref{AS}, { and then} we apply Young's inequality twice. This { yields}
$$
I \leq C R^{\frac 3s + \frac 13 - \frac {5p - 9}{s(2p - 3)}} \| \psi \nabla u \|_{L^p}^{\frac {3p}{s(2p - 3)}} \| \psi^3 u \|_{L^3}^{\frac {4sp - 3p - 6s}{s(2p - 3)}}
$$
$$
+ C R^{\frac 3s + \frac 13 - \frac {5p - 9}{s(2p - 3)}} \| \psi \nabla u \|_{L^p} \left( (R - \rho)^{-1} \| \psi^p u \|_{L^p} \right)^{\frac {3p + 3s - 2sp}{s(2p - 3)}} \| \psi^3 u \|_{L^3}^{\frac {4sp - 3p - 6s}{s(2p - 3)}}
$$
$$
\leq C R \| \psi \nabla u \|_{L^p}^{\frac {9p}{2sp + 3p - 3s}} + C R \| \psi \nabla u \|_{L^p}^{\frac {3s(2p - 3)}{2sp + 3p - 3s}} \left( (R - \rho)^{-1} \| \psi^p u \|_{L^p} \right)^{\frac {9p + 9s - 6sp}{2sp + 3p - 3s}} 
$$
$$
+ \frac 14 \int |\psi^3 u|^3 \mathrm{d}x
$$
\begin{equation}\label{3PI1}
\leq C R \| \psi \nabla u \|_{L^p}^{\frac {9p}{2sp + 3p - 3s}} + C R \left( (R - \rho)^{-1} \| \psi^p u \|_{L^p} \right)^{\frac {9p}{2sp + 3p - 3s}} + \frac 14 \int |\psi^3 u|^3 \mathrm{d}x.
\end{equation}

{ Secondly, we estimate $II$}. First, let $s \geq 3$. If we apply H\"{o}lder's inequality to $II$ and use \eqref{AS} along with Young's inequality, it follows
$$
II \leq C (R - \rho)^{-1} R \left( \fint_{B(R)} |\vect{V} - \vect{V}_{B(R)} |^s \mathrm{d}x \right)^{\frac 1s} \| \psi^3 u \|_{L^{3}}^2
$$
\begin{equation}\label{3PII1}
\leq C (R - \rho)^{-3} R^{4 - \frac {3(5p - 9)}{s(2p - 3)}} + \frac 14 \int |\psi^3 u|^3 \mathrm{d}x.
\end{equation}
\eqref{3PI1} and \eqref{3PII1} shows \eqref{3P3S}.
Now we let $s < 3$. Since
$$
\frac 32 < \frac s{s-1} \leq \frac {3p}{6 - p} < \frac {3p}{2(3 - p)},
$$
the { interpolation} inequality and Sobolev inequality imply that
$$
\| | \psi^4 u |^2 \|_{L^{\frac s{s-1}}} \leq \| | \psi^4 u |^2 \|_{L^{\frac {3p}{2(3-p)}}}^{\frac {p(3-s)}{2s(2p-3)}} \| | \psi^4 u |^2 \|_{L^{\frac 32}}^{\frac {5sp - 3p - 6s}{2s(2p - 3)}}
$$
$$
\leq \| \psi^4 u \|_{L^{\frac {3p}{3-p}}}^{\frac {p(3-s)}{s(2p - 3)}} \| \psi^4 u \|_{L^3}^{\frac {5sp - 3p - 6s}{s(2p-3)}}
$$
$$
\leq C \left( \| \psi^4 \nabla u \|_{L^p} + (R - \rho)^{-1} \| \psi^3 u \|_{L^p} \right)^{\frac {p(3-s)}{s(2p-3)}} \| \psi^4 u \|_{L^3}^{\frac {5sp - 3p - 6s}{s(2p - 3)}}
$$
$$
\leq C \| \psi \nabla u \|_{L^p}^{\frac {p(3-s)}{s(2p-3)}} \| \psi^3 u \|_{L^3}^{\frac {5sp - 3p - 6s}{s(2p-3)}} + C \left( (R - \rho)^{-1} \| \psi^p u \|_{L^p} \right)^{\frac {p(3-s)}{s(2p - 3)}} \| \psi^3 u \|_{L^3}^{\frac {5sp - 3p - 6s}{s(2p-3)}}.
$$
Similarly as above, { we obtain}
$$
II \leq C (R - \rho)^{-1} R^{\frac 3s + \frac 13 - \frac {5p - 9}{s(2p - 3)}} \| \psi \nabla u \|_{L^p}^{\frac {p(3-s)}{s(2p-3)}} \| \psi^3 u \|_{L^3}^{\frac {5sp - 3p - 6s}{s(2p - 3)}}
$$
$$
+ C (R - \rho)^{-1} R^{\frac 3s + \frac 13 - \frac {5p - 9}{s(2p - 3)}} \left( (R - \rho)^{-1} \| \psi^p u \|_{L^p} \right)^{\frac {p(3-s)}{s(2p-3)}} \| \psi^3 u \|_{L^3}^{\frac {5sp - 3p - 6s}{s(2p - 3)}}
$$
$$
\leq C (R - \rho)^{-\frac {3s(2p - 3)}{sp + 3p - 3s}} R^{\frac {2sp + 3p - 3s}{sp + 3p - 3s}} \| \psi \nabla u \|_{L^p}^{\frac {3p(3 - s)}{sp + 3p - 3s}}
$$
\begin{equation}\label{3PII2}
+ C (R - \rho)^{-\frac {3s(2p - 3)}{sp + 3p - 3s}} R^{\frac {2sp + 3p - 3s}{sp + 3p - 3s}} \left( (R - \rho)^{-1} \| \psi^p u \|_{L^p} \right)^{\frac {3p(3 - s)}{sp + 3p - 3s}} + \frac 14 \int |\psi^3 u|^3 \mathrm{d}x.
\end{equation}
By \eqref{3PI1} and \eqref{3PII2} we complete the proof.
\end{proof}

\section{Proof of Theorem~\ref{thm1}}

We assume all conditions for Theorem~\ref{thm1} { are fulfilled}. Note that in the case of $\frac {3p}{2p - 3} < s < + \infty$, we have by Jensen's inequality and \eqref{AS}
$$
\left( \frac 1{| B(r) |} \int_{B(r)} |\vect{V} - \vect{V}_{B(r)} |^{\frac {3p}{2p - 3}} \mathrm{d}x \right)^{\frac {2p - 3}{3p}} \leq \left( \frac 1{| B(r) |} \int_{B(r)} |\vect{V} - \vect{V}_{B(r)} |^s \mathrm{d}x \right)^{\frac 1s} \leq C r^{\frac 3p - \frac 43}.
$$
This { shows  that} we can { reduce \eqref{AS} to  that case of} $s = \frac {3p}{2p - 3}$. 
{ Hence, in general we may  ristrict the range of $s$ to}
$$
\frac {3p}{2(2p - 3)} \leq s \leq \frac {3p}{2p - 3}, \qquad s > \frac {9 - 3p}{2p - 3}.
$$

Let $1 < r < + \infty$ be arbitrarily chosen. We set $r \leq \rho < R \leq 4r$ and $\overline{R} = \frac {R + \rho}2$. The first claim is that
\begin{equation}\label{EI}
\int_{\R^3} \left| \nabla u \right|^p \mathrm{d}x < + \infty.
\end{equation}
Let $\zeta \in C^{\infty}_{c}(B(\overline{R}))$ be a radially non-increasing function such that $0 \leq \zeta \leq 1$, $\zeta = 1$ on $B(\rho)$ and $|\nabla \zeta| \leq C (R - \rho)^{-1}$ for some $C > 0$. If we insert $\phi = \zeta^p$ into \eqref{LEI}, then we have
$$
\int_{B(\overline{R})} |\vect{D}(u)|^p \zeta^p \mathrm{d}x = - \int_{B(\overline{R})} | \vect{D}(u) |^{p - 2} \vect{D}(u) : u \otimes \nabla \zeta^p \mathrm{d}x
$$
$$
+ \frac 12 \int_{B(\overline{R})} | u |^2 u \cdot \nabla \zeta^p \mathrm{d}x + \int_{B(\overline{R})} (\uppi - \uppi_{B(\overline{R})}) u \cdot \nabla \zeta^p  \mathrm{d}x.
$$
H\"{o}lder's inequality and Young's inequality imply that
$$
\int_{B(\overline{R})} |\vect{D}(u)|^p \zeta^p \mathrm{d}x \leq C \int_{B(\overline{R})} |u|^p |\nabla \zeta|^p 
\mathrm{d}x
$$
$$
+ C \int_{B(\overline{R})} | u |^3 |\nabla \zeta| |\zeta|^{p-1} \mathrm{d}x + C \int_{B(\overline{R})} |\uppi - \uppi_{B(\overline{R})}|^{\frac 32} |\nabla \zeta| |\zeta|^{p-1} \mathrm{d}x.
$$
Employing Calderón-Zygmund's inequality, we obtain
\begin{equation}\label{CZ}
\int \left| \nabla ( u \zeta ) \right|^p \mathrm{d}x \leq C \int \left| { \vect{D}(u)} \right|^p \zeta^p \mathrm{d}x + C \int \left| u \right|^p \left| \nabla \zeta \right|^p \mathrm{d}x.
\end{equation}
{ Using \eqref{CZ} along  with  \eqref{PE}}, it follows
$$
\int_{B(\rho)} |\nabla u|^p \mathrm{d}x \leq C (R - \rho)^{-p} \int_{B(\overline{R})} |u|^p \mathrm{d}x + C (R - \rho)^{-1} \int_{B(\overline{R})} | u |^3 \mathrm{d}x
$$
$$
+ C (R - \rho)^{-1} R^{\frac 32 \left( \frac3p - 1 \right)} \left( \int_{B(R)} \left| \nabla u \right |^p \mathrm{d}x \right)^{\frac 32 \left(1 - \frac 1p \right)}.
$$
We consider $\psi \in C^{\infty}_c(B(R))$ a radially non-increasing function satisfying $0 \leq \psi \leq 1$, $\psi = 1$ on $B(\overline{R})$ and $|\nabla \psi| \leq C (R - \rho)^{-1}$ for some $C > 0$. By { the properties} of $\psi$ we have that
$$
\int_{B(\rho)} |\nabla u|^p \mathrm{d}x \leq C (R - \rho)^{-p} \int_{B(R)} |\psi^p u|^p \mathrm{d}x + C (R - \rho)^{-1} \int_{B(R)} |\psi^3 u|^3 \mathrm{d}x
$$
\begin{equation}\label{123}
+ C (R - \rho)^{-1} R^{\frac 32 \left( \frac3p - 1 \right)} \left( \int_{B(R)} \left| \nabla u \right |^p \mathrm{d}x \right)^{\frac 32 \left(1 - \frac 1p \right)} = I + II + III.
\end{equation}

Let $\epsilon > 0$ be an arbitrary real number. Before calculating $II$ first, we note that $\psi$ satisfies the assumptions for Lemma~\ref{LEMPP} and Lemma~\ref{LEM3P}. Observing that for $s > \frac {9 - 3p}{2p - 3}$
$$
\frac {9p}{2sp + 3p - 3s} < p,
$$
{ we may apply} Young's inequality to \eqref{3P3S} for $s \geq 3$. This yields
$$
II \leq C(\epsilon) (R - \rho)^{-\frac {2sp + 3p - 3s}{2sp - 3s + 3p - 9}} R^{\frac {2sp + 3p - 3s}{2sp - 3s + 3p - 9}} + I
$$
$$
+ C (R - \rho)^{-4} R^{4 - \frac {3(5p - 9)}{s(2p - 3)}} + \epsilon \int_{B(R)} |\psi \nabla u|^p \mathrm{d}x.
$$
{ We continue  estimating $I$}. Since
$$
4 - \frac {3(5p - 9)}{s(2p - 3)} < 4,
$$
{ we get for $R > 1$}
$$
II \leq C(\epsilon) (R - \rho)^{-\frac {2sp + 3p - 3s}{2sp - 3s + 3p - 9}} R^{\frac {2sp + 3p - 3s}{2sp - 3s + 3p - 9}} + C (R - \rho)^{-4} R^4 + \epsilon \int_{B(R)} |\psi \nabla u|^p \mathrm{d}x.
$$
In the case of $s < 3$, { we  see that} for $s > \frac {9 - 3p}{2p - 3}$ { it holds}
$$
\frac {3p(3 - s)}{sp - 3s + 3p} < \frac {3p(3 - s)}{s(2p - 3) - 3s + 3p} < p.
$$
Thus \eqref{3PS3} and Young's inequality give
$$
II \leq C(\epsilon) (R - \rho)^{-\frac {2sp + 3p - 3s}{2sp - 3s + 3p - 9}} R^{\frac {2sp + 3p - 3s}{2sp - 3s + 3p - 9}} + I
$$
$$
+ C(\epsilon) (R - \rho)^{-\frac {7sp + 3p - 12s}{sp + 3p - 9}} R^{\frac {2sp + 3p - 3s}{sp + 3p - 9}} + \epsilon \int_{B(R)} |\psi \nabla u|^p \mathrm{d}x.
$$
Since
$$
\frac {2sp + 3p - 3s}{sp + 3p - 9} < \frac {7sp + 3p - 12s}{sp + 3p - 9},
$$
$R > 1$ shows that
$$
II \leq C(\epsilon) (R - \rho)^{-\frac {2sp + 3p - 3s}{2sp - 3s + 3p - 9}} R^{\frac {2sp + 3p - 3s}{2sp - 3s + 3p - 9}} + C(\epsilon) (R - \rho)^{-\frac {7sp + 3p - 12s}{sp + 3p - 9}} R^{\frac {7sp + 3p - 12s}{sp + 3p - 9}} 
$$
$$
+ \epsilon \int_{B(R)} |\psi \nabla u|^p \mathrm{d}x.
$$
Hence,  { in both cases we obtain the following estimate}
$$
II \leq C(\epsilon) (R - \rho)^{-\frac {7sp + 3p - 12s}{sp + 3p - 9}} R^{\frac {7sp + 3p - 12s}{sp + 3p - 9}} + C(\epsilon) (R - \rho)^{-\frac {2sp + 3p - 3s}{2sp - 3s + 3p - 9}} R^{\frac {2sp + 3p - 3s}{2sp - 3s + 3p - 9}}
$$
\begin{equation}\label{EII1}
+ \epsilon \int_{B(R)} |\psi \nabla u|^p \mathrm{d}x. 
\end{equation}

Now we estimate $I$. If $s \geq p$, using \eqref{PPPS} and Young's inequality, { we see that}
$$
I \leq C(\epsilon) (R - \rho)^{-2p} R^{3 + \frac p3 - \frac {p(5p - 9)}{s(2p - 3)}} + \epsilon \int_{B(R)} |\psi \nabla u|^p \mathrm{d}x.
$$
Since
\begin{equation*}
3 + \frac p3 - \frac {p(5p - 9)}{s(2p - 3)} \leq 6 - \frac {4p}3
\end{equation*}
for $s \leq \frac {3p}{2p - 3}$, $R > 1$ implies
\begin{equation}\label{SISP}
I \leq C(\epsilon) (R - \rho)^{-2p} R^{6 - \frac {4p}3} + \epsilon \int_{B(R)} |\psi \nabla u|^p \mathrm{d}x
\end{equation}
$$
\leq C(\epsilon) (R - \rho)^{-2p} R^{2p} + \epsilon \int_{B(R)} |\psi \nabla u|^p \mathrm{d}x.
$$
{ Notice that for $s < p$, }
$$
0 < \frac {p(sp - 3s + 3p)}{2sp - 3s + 3p} < p
$$
and
$$
0 < \frac {p(3p - 3s)}{sp - 3s + 3p} < p.
$$
By applying Young's inequality to \eqref{PPSP} { we get}
\begin{equation}\label{SIPS}
I \leq C(\epsilon) (R - \rho)^{- 2p + 3 - \frac {3p}s} R^{\frac p3 + \frac {p^2}{s(2p - 3)}} + \epsilon \int_{B(R)} |\psi \nabla u|^p \mathrm{d}x.
\end{equation}
Since
\begin{equation}\label{EST1}
\frac p3 + \frac {p^2}{s(2p - 3)} < 2p - 3 + \frac {3p}s \leq 6p - 9
\end{equation}
for $s \geq \frac {3p}{2(2p - 3)}$, $R > 1$ and $R(R - \rho)^{-1} > 1$ imply that
$$
I \leq C(\epsilon) (R - \rho)^{-(6p - 9)} R^{6p - 9} + \epsilon \int_{B(R)} |\psi \nabla u|^p \mathrm{d}x.
$$
Therefore, in each case we obtain
\begin{equation}\label{EI1}
I \leq  C(\epsilon) (R - \rho)^{-2p} R^{2p} + C(\epsilon) (R - \rho)^{-(6p - 9)} R^{6p - 9} + \epsilon \int_{B(R)} |\psi \nabla u|^p \mathrm{d}x.
\end{equation}

Applying Young's inequality to $III$, { we find}
$$
III \leq C(\epsilon) (R - \rho)^{-\frac {2p}{3-p}} R^3 + \epsilon \int_{B(R)} |\psi \nabla u|^p \mathrm{d}x.
$$
By $3 < \frac {2p}{3-p}$ and $R > 1$, it follows that
\begin{equation}\label{EIII1}
III \leq C(\epsilon) (R - \rho)^{-\frac {2p}{3-p}} R^{\frac {2p}{3-p}} + \epsilon \int_{B(R)} |\psi \nabla u|^p \mathrm{d}x.
\end{equation}
We define
$$
\gamma := \max \left\{ \frac {7sp + 3p - 12s}{sp + 3p - 9}, \frac {2sp + 3p - 3s}{2sp - 3s + 3p - 9}, 2p, 6p - 9, \frac {2p}{3-p} \right\}.
$$
From \eqref{EII1}, \eqref{EI1}, \eqref{EIII1} and $R(R - \rho)^{-1} > 1$ { we deduce  that}
$$
I + II + III \leq C(\epsilon) (R - \rho)^{-\gamma} R^{\gamma} + \epsilon \int_{B(R)} |\nabla u|^p \mathrm{d}x.
$$
{ Inserting this estimate  into} \eqref{123} and applying the iteration Lemma in \cite[Lemma~3.1]{G83} for sufficiently small $\epsilon > 0$, we are led to
$$
\int_{B(\rho)} \left| \nabla u \right|^p \mathrm{d}x \leq C (R - \rho)^{-\gamma} R^{\gamma}.
$$
By taking $R = 2r$, $\rho = r$ and passing $r \rightarrow + \infty$, we obtain \eqref{EI}.

Secondly we claim that
\begin{equation}\label{LPo1}
r^{-p} \int_{B(2r) \setminus B(r)} \left| u \right|^p \mathrm{d}x = o(1) \qquad \mbox{as} \qquad r \rightarrow + \infty.
\end{equation}
We consider a cut-off function $\psi \in C^{\infty}_c (B(4r) \setminus B(\frac r2))$ satisfying $0 \leq \psi \leq 1$, $\psi = 1$ on $B(2r) \setminus B(r)$ and $|\nabla \psi| \leq Cr^{-1}$. Then $\psi$ satisfies the assumptions for Lemma~\ref{LEMPP} when $R = 4r$ and $\rho = r$. Hence, in the case of $s \geq p$ we use \eqref{SISP} to obtain
$$
r^{-p} \int_{B(4r)} |\psi^p u|^p \mathrm{d}x \leq C r^{6 - \frac {10p}3} + C \int_{B(4r)} |\psi \nabla u|^p \mathrm{d}x 
$$
$$
\leq C r^{6 - \frac {10p}3} + C \int_{B(4r) \setminus B(\frac r2)} |\nabla u|^p \mathrm{d}x.
$$
If $s < p$, using \eqref{SIPS}, we have that
$$
r^{-p} \int_{B(4r)} |\psi^p u|^p \mathrm{d}x \leq C r^{\frac p3 + \frac {p^2}{s(2p - 3)} - \left( 2p - 3 + \frac {3p}s \right)} + C \int_{B(4r) \setminus B(\frac r2)} |\nabla u|^p \mathrm{d}x.
$$
Thus, observing \eqref{EST1} and \eqref{EI}, we obtain
\begin{equation}\label{SILPo1}
r^{-p} \int_{B(4r)} |\psi^p u|^p \mathrm{d}x = o(1) \qquad \mbox{as} \qquad r \rightarrow + \infty
\end{equation}
which implies \eqref{LPo1}.

Next, we claim
\begin{equation}\label{L3o1}
r^{-1} \int_{B(2r) \setminus B(r)} \left| u \right|^3 \mathrm{d}x = o(1) \qquad \mbox{as} \qquad r \rightarrow + \infty
\end{equation}
and
\begin{equation}\label{L3O1}
r^{-1} \int_{B(r)} \left| u \right|^3 \mathrm{d}x = O(1) \qquad \mbox{as} \qquad r \rightarrow + \infty.
\end{equation}
We set the same function $\psi \in C^{\infty}_c (B(4r) \setminus B(\frac r2))$ with $R = 4r$ and $\rho = r$. For $s \geq 3$ we can use \eqref{3P3S} { to infer}
$$
r^{-1} \int_{B(4r)} |\psi^3 u|^3 \mathrm{d}x 
$$
$$
\leq C \bigg( \int_{B(4r) \setminus B(\frac r2)} |\nabla u|^p \mathrm{d}x \bigg)^{\frac {9}{2sp + 3p - 3s}} + C \bigg( r^{-p} \int_{B(4r)} |\psi^p u|^p \mathrm{d}x \bigg)^{\frac {9}{2sp + 3p - 3s}} + C r^{- \frac {3(5p - 9)}{s(2p - 3)}}.
$$
{ In case} $s < 3$, \eqref{3PS3} gives that
$$
r^{-1} \int_{B(4r)} |\psi^3 u|^3 \mathrm{d}x 
$$
$$
\leq C \bigg( \int_{B(4r) \setminus B(\frac r2)} |\nabla u|^p \mathrm{d}x \bigg)^{\frac {9}{2sp + 3p - 3s}} + C \bigg( r^{-p} \int_{B(4r)} |\psi^p u|^p \mathrm{d}x \bigg)^{\frac {9}{2sp + 3p - 3s}}
$$
$$
+ C r^{\frac {s(9 - 5p)}{sp + 3p - 3s}} \bigg( \int_{B(4r) \setminus B(\frac r2)} |\nabla u|^p \mathrm{d}x \bigg)^{\frac {3(3 - s)}{sp + 3p - 3s}} + C r^{\frac {s(9 - 5p)}{sp + 3p - 3s}} \bigg( r^{-p} \int_{B(4r)} |\psi^p u|^p \mathrm{d}x \bigg)^{\frac {3(3 - s)}{sp + 3p - 3s}}.
$$
In each case \eqref{EI} and \eqref{SILPo1} imply \eqref{L3o1}.

To verify \eqref{L3O1} we choose $r_0 > 1$ arbitrarily and let $r > r_0$. Then for $j \in \N$ satisfying $r < 2^j r_0$, we have
$$
r^{-1} \int_{B(r)} \left| u \right|^3 \mathrm{d}x \leq r^{-1} \int_{B(2r) \setminus B(r)} \left| u \right|^3 \mathrm{d}x + \frac 12 \bigg( \left( \frac r2 \right)^{-1} \int_{B(r) \setminus B(\frac r2)} \left| u \right|^3 \mathrm{d}x \bigg) + \ldots
$$
$$
+ \frac 1{2^j} \bigg( \left( \frac r{2^j} \right)^{-1} \int_{B(\frac r{2^{j - 1}}) \setminus B(\frac r{2^j})} |u|^3 \mathrm{d}x \bigg) + r^{-1} \int_{B(r_0)} |u|^3 \mathrm{d}x
$$
$$
\leq 2 \sup_{r_0 \leq \widetilde{r} < + \infty} \left\{ \widetilde{r}^{-1} \int_{B(2\widetilde{r}) \setminus B(\widetilde{r})} \left| u \right|^3 \mathrm{d}x \right\} + r^{-1} \int_{B(r_0)} |u|^3 \mathrm{d}x.
$$
Then for sufficiently large $r_0 > 1$, \eqref{L3O1} is obtained by \eqref{L3o1} and 
{ $u \in W^{1,p}(B(r_{0})) \hookrightarrow L^3(B(r_{0}))$}.\\

\begin{pfthm1}
Let $\psi \in C^{\infty}_{c}(B(2r))$ satisfy $0 \leq \psi \leq 1$, $\psi = 1$ on $B(r)$ and $|\nabla \psi| \leq C r^{-1}$. 
{ We observe   \eqref{LEI} with  $\psi^2 = \phi$,  and apply H\"{o}lder's inequality, Young's inequality and Calderón-Zygmund's inequality to get }
$$
\int_{B(r)} |\nabla u|^p \mathrm{d}x \leq C \int_{B(2r)} |u|^p |\nabla \phi|^p \mathrm{d}x + C \int_{B(2r)} |u|^3 |\nabla \phi| \mathrm{d}x
$$
$$
+ C \int_{B(2r)} |\uppi - \uppi_{B(2r)}||u| |\nabla \phi| \mathrm{d}x
$$
$$
\leq C r^{-p} \int_{B(2r) \setminus B(r)} |u|^p \mathrm{d}x + C r^{-1} \int_{B(2r) \setminus B(r)} | u |^3 \mathrm{d}x 
$$
$$
+ C r^{-1} \int_{B(2r) \setminus B(r)} |\uppi - \uppi_{B(2r)}||u| \mathrm{d}x = IV + V + VI.
$$
The properties \eqref{LPo1} and \eqref{L3o1} directly shows that
$$
IV + V \rightarrow 0 \qquad \mbox{as} \qquad r \rightarrow + \infty.
$$
To estimate $VI$ we use H\"{o}lder's inequality and \eqref{PE} when $s = \frac 32$. This { yields}
$$
VI \leq C \left( r^{-1} \int_{B(2r)} |\pi - \pi_{B(2r)}|^{\frac 32} \mathrm{d}x \right)^{\frac 23} \left( r^{-1} \int_{B(2r) \setminus B(r)} | u |^3 \mathrm{d}x \right)^{\frac 13}
$$
$$
\leq C \left( r^{-1} \int_{B(2r)} |\nabla u|^{\frac {3(p-1)}2} \mathrm{d}x + r^{-1} \int_{B(2r)} |u|^3  \mathrm{d}x \right)^{\frac 23} \left( r^{-1} \int_{B(2r) \setminus B(r)} | u |^3 \mathrm{d}x \right)^{\frac 13}.
$$
{ According to  \eqref{L3O1}} it is sufficient to show that
$$
r^{-1} \int_{B(2r)} |\nabla u|^{\frac {3(p-1)}2} \mathrm{d}x = { O(1)} \qquad \mbox{as} \qquad r \rightarrow + \infty.
$$
On the other hand, H\"{o}lder's inequality with $\frac {3(p-1)}2 < p$ implies that
$$
r^{-1} \int_{B(2r)} |\nabla u|^{\frac {3(p-1)}2} \mathrm{d}x \leq r^{-1} |B(2r)|^{\frac 3{2p} - \frac 12} \left( \int_{B(2r)} |\nabla u|^p \mathrm{d}x \right)^{\frac {3(p - 1)}{2p}}
$$
$$
= r^{\frac 9{2p} - \frac 52} \left( \int_{B(2r)} |\nabla u|^p \mathrm{d}x \right)^{\frac {3(p - 1)}{2p}}.
$$
This implies
$$
VI \rightarrow 0 \qquad \mbox{as} \qquad r \rightarrow + \infty
$$
and
$$
\int_{B(r)} \left| \nabla u \right|^p = o(1) \qquad \mbox{as} \qquad r \rightarrow + \infty.
$$
Accordingly, $u \equiv const$ and by means of \eqref{L3o1}, we have that $u \equiv 0$.
\end{pfthm1}

\hspace{0.5cm}
$$\mbox{\bf Acknowledgements}$$
Chae's research  was partially supported by NRF grants 2021R1A2C1003234, and by the Chung-Ang University research grant in 2019.
Wolf has been { supported by }NRF grants 2017R1E1A1A01074536.
The authors declare that they have no conflict of interest.

\end{document}